\documentclass{amsart}

\usepackage{graphicx,epsfig}

\usepackage{amsmath,amssymb,latexsym, amsfonts, amscd, amsthm, xy}

\usepackage{url}

\usepackage{hyperref}

\usepackage{mathrsfs}

\input{xy}
\xyoption{all}

\usepackage[usenames]{color}

\makeindex 

=630pt \headheight = 12pt \headsep = 20pt

\hypersetup{
    colorlinks   = true,
      citecolor    = red
}

\hypersetup{linkcolor=red}

\hypersetup{linkcolor=blue}

\theoremstyle{plain}
\newtheorem{theorem}{Theorem}[section]

\newtheorem{proposition}[theorem]{Proposition}
\newtheorem{corollary}[theorem]{Corollary}

\newtheorem{lemma}[theorem]{Lemma}

\theoremstyle{definition}

\newtheorem{remark}[theorem]{Remark}

\def\bF{\mathbb{F}}

\def\bQ{\mathbb{Q}}
\def\bZ{\mathbb{Z}}

\def\cG{\mathcal{G}}

\def\fm{\mathfrak{m}}

\def\deg{\mathbf{deg}}

\def\Int{\mathrm{Int}}

\begin{document}

\title[Generating polynomials for generalized binomial coefficients]{On the generating polynomials for the distribution of generalized binomial coefficients in discrete valuation domains}

\author{Dong Quan Ngoc Nguyen}

\date{June 12, 2020}

\address{Department of Applied and Computational Mathematics and Statistics \\
         University of Notre Dame \\
         Notre Dame, Indiana 46556, USA }

\email{\href{mailto:dongquan.ngoc.nguyen@nd.edu}{\tt dongquan.ngoc.nguyen@nd.edu}}

\urladdr{http://nd.edu/~dnguye15}

\maketitle

\tableofcontents

\begin{abstract}

For a discrete valuation domain $V$ with maximal ideal $\mathfrak{m}$ such that the residue field $V/\mathfrak{m}$ is finite, there exists a sequence of polynomials $(F_n(x))_{n \ge 0}$ defined over the quotient field $K$ of $V$ that forms a basis of the $V$-module $\text{Int}(V) = \{f \in K[x] | f(V)\subseteq V\}$. This sequence of polynomials bears many resemblances to the classical binomial polynomials $(\binom{x}{n})_{n \ge 0}$. In this paper, we introduce a generating polynomial to account for the distribution of the $V$-values of the polynomials $F_n(x)$ modulo the maximal ideal $\mathfrak{m}$, and prove a result that provides a method for counting exactly how many $V$-values of the polynomials $(F_n(x))_{n \ge 0}$ fall into each of the residue classes modulo $\mathfrak{m}$. Our main theorem in this paper can be viewed as an analogue of the classical theorem of Garfield and Wilf in the context of discrete valuation domains.

\end{abstract}

\section{Introduction}

For an integer $n \ge 0$, the classical binomial polynomial $\binom{x}{n} \in \bQ[x]$ is defined as 
\begin{align*}
\dfrac{x(x - 1)(x - 2)\cdots (x - n + 1)}{n!}. 
\end{align*}
The sequence $\left(\binom{x}{n}\right)_{n \ge 0}$ plays an important role in studying the $\bZ$-module $\text{Int}(\bZ) = \{f \in \bQ[x] \; | \; f(\bZ) \subseteq \bZ\}$. Studying the $\bZ$-values of these polynomials $\binom{x}{n}$ at integers $m \ge 0$ thus attracts special attention, and is an  active research area. Many basic questions arise concerning with such $\bZ$-values; for example, for a given prime $p$ and an integer $n \ge 0$, can one describe exactly which integers $m$ are such that the values of binomial polynomial $\binom{x}{n}$ at $m$ are divisible by $p$, or satisfy certain congruence conditions modulo $p$? Motivated by such question, Garfield and Wilf \cite{GW} proposed a method using generating polynomials for studying the distribution of $\bZ$-values of binomial polynomials modulo primes. 

There are many strong analogues between the integers $\bZ$ and the ring of polynomials over a finite field $\bF_q$, say $\bF_q[t]$ (see, for example, Goss \cite{Goss} and Weil \cite{W}). Inspired by the work of Garfield and Wilf \cite{GW} and the analogies between $\bZ$ and $\bF_q[t]$, the author \cite{N} proved a function field analogue of the Garfield--Wilf theorem with the classical binomial coefficients replaced by the Carlitz binomial coefficients (see Carlitz \cite{Car}). 

In this paper, we further consider another analogy between $\bZ$ and an arbitrary discrete valuation domain with finite residue field in view of the classical theorem of Garfield and Wilf. In such a discrete valuation domain $V$ with maximal ideal $\fm$, there exists a sequence of polynomials $(F_n(x))_{n \ge 0}$ (see (\ref{def-F_n}) below for a precise notion of $F_n(x)$) defined over the quotient field $K$ of $V$ such that they form a basis of the $V$-module $\text{Int}(V) = \{f \in K[x] \; | \; f(V) \subseteq V\}$ in a similar manner as the classical binomial polynomials $(\binom{x}{n})_{n \ge 0}$ do for the $\bZ$-module $\text{Int}(\bZ)$. In view of this similarity, one can view $(F_n(x))_{n \ge 0}$ as an analogue in $V$ of the classical binomial polynomials $(\binom{x}{n})_{n \ge 0}$. One can ask for analogues of results about the classical binomial polynomials in the context of $V$ with $(F_n(x))_{n \ge 0}$ in place of the classical binomial polynomials; for illustration, a basic problem is to understand the divisibility properties of the $V$-values of $F_n(x)$ at elements $u \in V$. Regarding such a problem, it suffices to consider only the $V$-values of $F_n(x)$ at elements $u_n \in V$, where $(u_n)_{n \ge 0}$ forms a \textit{very well-distributed and well-ordered sequence} (see \cite[Definition {\bf II.2.1}]{CC}) in $V$. Indeed, for such a sequence $(u_n)_{n \ge 0}$, it is well-known that a polynomial $f \in K[x]$ of degree $n$ belongs in $\text{Int}(V)$ if and only if the values of $f(x)$ at $u_0, u_1, \ldots, u_n$ belong in $V$ (see \cite[Corollary {\bf II2.8}]{CC}). Since the sequence $(F_n(x))_{n \ge 0}$ forms a basis of the $V$-module $\text{Int}(V)$, it suffices to study the $V$-values of $(F_n(x))_{n \ge 0}$ at the sequence $(u_n)_{n \ge 0}$. Based on this similarity with the $\bZ$-values of the classical binomial polynomials $\binom{x}{n}$ at integers $m \ge 0$, we call the values $F_n(u_m)$ for $n, m \ge 0$ the \textit{generalized binomial coefficients} in $V$ in analogy with the classical binomial coefficients $\binom{m}{n}$. 

Motivated by the classical theorem of Garfield and Wilf \cite{GW} and the above discussion about analogies between the classical binomial polynomials in $\bZ$ and the polynomials $(F_n(x))_{n \ge 0}$ in $V$, it is natural to ask whether there is an analogue of Garfield--Wilf theorem in the context of discrete valuation domains. We give an affirmative answer to this question, and our main theorem (see Theorem \ref{main-thm}) in this paper can be viewed as an analogue of the theorem of Garfield and Wilf in the context of discrete valuation domains. 

Our paper is structured as follows. In Section \ref{sec-basic-notions}, we introduce some basic notions and notation that will be used throughout the paper. Furthermore, in the same section, we explain an analogue of the classical binomial polynomials in the setting of discrete valuation domains, and recall an analogue of Lucas' theorem for discrete valuation domains which is due to Boulanger and Chabert \cite{BC}. In Section \ref{section-main-result}, we prove the main theorem in this paper (see Theorem \ref{main-thm}). The proof of our main theorem is modeled on that of the classical theorem of Garfield and Wilf. Our main contribution in this paper is to find a correct analogue of the generating polynomials for counting the residue classes of the integral values of the polynomials $(F_n(x))_{n \ge 0}$ modulo the maximal ideal of a discrete valuation domain. Note that despite similarities between the classical binomial coefficients $\binom{m}{n}$ and the generalized binomial coefficients $F_n(u_m)$, there is a distinct difference between these two values: \textit{both arguments $n, m$ in the classical binomial coefficients $\binom{m}{n}$ belong in $\bZ$ whereas one argument, say $u_m$ in the generalized binomial coefficients $F_n(u_m)$ belong in a discrete valuation domain $V$ which in general does not contain any ordering as $\bZ$ does}. Hence in order to introduce a notion of generating polynomials for the distribution of the generalized binomial coefficients $F_n(u_m)$, we need to introduce two basic maps (see Subsection \ref{subsec-2maps}) to transfer elements between $V$ and $\bZ$ so that one can properly \textit{count} the number of generalized binomial coefficients $F_n(u_m)$ modulo the maximal ideal of a discrete valuation domain. 

Throughout the paper, $\bZ$ denotes the set of integers, $\bZ_{\ge 0}$ stands for the set of nonnegative integers, and $\bZ_{>0}$ is the set of positive integers.

\section{Basic notions and notation} 
\label{sec-basic-notions}

\subsection{An analogue of Lucas' theorem for discrete valuation domains}

In this subsection, we introduce a few basic notions and notations that we will use throughout this paper. We also recall an analogue of Lucas's theorem for discrete valuation domains which was proved by Boulanger and Chabert \cite{BC}. 

Let $V$ be a discrete valuation domain, and let $\fm$ be the maximal ideal of $V$. Throughout this paper, we assume that the residue field $V/\fm$ is finite, and isomorphic to the finite field $\bF_q$ for some $q$, which implies that the cardinality of $V/\fm$ is $q$. Let $\pi$ be a generator of $\fm$, $K$ the quotient field of $V$, and let $v$ the corresponding valuation of $K$. Let $\widehat V, \widehat K, \widehat \fm$ be the completions of $V$, $K$, and $\fm$ with respect to the $\fm$-adic topology, respectively. By abuse of notation, we still use the same notation $v$ to denote the extension of $v$ to $\widehat K$. 

We choose elements $u_0, \ldots, u_{q - 1}$ in $V$ such that $u_0 = 0$, and $R = \{u_0, u_1, \ldots, u_{q - 1}\}$ is the set of representatives of $V$ modulo $\fm$. It is known (see \cite[Chapter 2]{CF}) that each element $x \in \widehat V$ has a unique $\pi$-adic expansion of the form
$$x = \sum_{j \ge 0} x_j\pi^j,$$
where the $x_j$ are some elements in $R$.
We extend the set $R$ as follows.  For each integer $n \ge 0$, represent $n$ in the $q$-adic expansion of the form 
$$n = n_0 + n_1q + \cdots n_kq^k,$$
where the $n_i$ are integers such that $0 \le n_i \le q - 1$ for each $i$.

We define 
\begin{align}
\label{def-u_n}
u_n = u_{n_0} + u_{n_1}\pi + \cdots + u_{n_k}\pi^k \in V.
\end{align}

It is clear that for $0 \le n \le q - 1$, the element $u_n$ we just construct above is the same as in the set $R$. With the above construction, we obtain a sequence of elements $(u_n)_{n \ge 0}$ in $V$ whose first $q$ elements coincide with $R$. The sequence $(u_n)_{n \ge 0}$ forms a \textit{very well distributed and well-ordered sequence} in $V$ (see \cite[Definition {\bf II.2.1} and Proposition {\bf II.2.3}]{CC}).

We construct a sequence of polynomials $(F_n(x))_{n \ge 0} \subseteq K[x]$ (which can be viewed as an analogue of the classical binomial polynomisl) by letting, for each $n \ge 1$,
\begin{align}
\label{def-F_n}
F_n(x) = \prod_{k = 1}^{n - 1}\dfrac{x - u_k}{u_n - u_k},
\end{align}
and setting $F_0(x) = 1$.

The sequence $(F_n(x))_{n \ge 0}$ shares several similar properties as the sequence of classical binomial polynomials $\left(\binom{x}{n}\right)_{n \ge 0}$, where we recall that for each $n \ge 0$,
\begin{align*}
\binom{x}{n} = \dfrac{x(x - 1)(x - 2)\cdots (x - n + 1)}{n!}.
\end{align*}

One of the main resemblances between these two sequences (see  Theorem {\bf II.2.7} in \cite{CC}) is that the sequence $(F_n(x))_{n \ge 0}$ forms a basis of the $V$-module $\Int(V) = \{f \in K[x] \; | \; f(V) \subseteq V\}$, which is an analogue of the classical result that $\left(\binom{x}{n}\right)_{n \ge 0}$ is a basis of the $\bZ$-module $\Int(\bZ) = \{f \in \bQ[x] \; | \; f(\bZ) \subseteq \bZ\}$. The pair  $\{(u_n)_{n \ge 0}, (F_n(x))_{n \ge 0}\}$ play a similar role for understanding $\text{Int}(V)$ as the role of the pair $\{\bZ_{\ge 0}, (\binom{x}{n})_{n \ge 0}$ for studying $\text{Int}(\bZ)$. This can be seen by recalling the fact (see cite[Corollary {\bf II.2.8.}]{CC}) that if $f$ is a polynomial in $K[x]$ of degree $n$, then $f \in \text{Int}(V)$ if and only if all the values $f(u_0), \ldots, f(u_n)$ belong in $V$. Since $(F_n(x))_{n \ge 0}$ forms a basis of $\Int(V)$, the values of $F_n(u_m)$ for $n, m \in \bZ_{\ge 0}$ play an important role for studying the structure of $\text{Int}(V)$ in a similar manner as the classical binomial coefficients $\binom{m}{n}$ for understanding the structure of $\text{Int}(\bZ)$. In view of this analogy, it is natural to call the values of $F_n(u_m)$ for $n, m \in \bZ_{\ge 0}$ the \textit{generalized binomial coefficients} in $V$.

Another analogy between the sequence $(F_n(x))_{n \ge 0}$ and the classical binomial polynomials $\left(\binom{x}{n}\right)_{n \ge 0}$ which we need in the proof of our main theorem, is reflected in the following result which is due to Boulanger and Chabert (see \cite[Theorem {\bf2.2}]{BC}).

\begin{theorem}
\label{Lucas-thm}
(analogue of Lucas' theorem, see \cite[Theorem {\bf2.2}]{BC})

Let $x$ be an element of $\widehat V$ such that the $\pi$-adic expansion of $x$ is of the form
\begin{align*}
x = \sum_{j \ge 0} x_j\pi^j,
\end{align*}
where the $x_j$ are some elements in $R$. Let $n$ be a positive integer whose $q$-adic expansion is of the form
\begin{align*}
n = n_0 + n_1q + \cdots n_kq^k,
\end{align*}
where the $n_i$ are integers such that $0 \le n_i \le q - 1$ for each $i$. Then
\begin{align*}
F_n(x) \equiv F_{n_0}(x_0)F_{n_1}(x_1)\cdots F_{n_k}(x_k) \pmod{\widehat \fm}.
\end{align*}

\end{theorem}

Since $V \cap \widehat \fm = \fm$ (see \cite{Bou}), and the $F_n(x) \in V[x]$, the next result follows immediately from Theorem \ref{Lucas-thm}.

\begin{corollary}
\label{Lucas-cor}

Let $x$ be an element of $V$ such that the $\pi$-adic expansion of $x$ is of the form
\begin{align*}
x = \sum_{j \ge 0} x_j\pi^j,
\end{align*}
where the $x_j$ are some elements in $R$. Let $n$ be a positive integer whose $q$-adic expansion is of the form
\begin{align*}
n = n_0 + n_1q + \cdots n_kq^k,
\end{align*}
where the $n_i$ are integers such that $0 \le n_i \le q - 1$ for each $i$. Then
\begin{align*}
F_n(x) \equiv F_{n_0}(x_0)F_{n_1}(x_1)\cdots F_{n_k}(x_k) \pmod{\fm}.
\end{align*}

\end{corollary}

\subsection{A semigroup structure on $R^{\omega}$} 
\label{subsec-semigroup}

In this subsection, we introduce a \textit{semigroup} structure on $R^{\omega} = \bigcup_{i =1}^{\infty} R^i$, where for each $i \ge 1$, $R^i = \underbrace{R \times R \times \cdots \times R}_{\text{$i$ copies}}$. Each element $\hat \alpha \in R^{\omega}$ can be uniquely written in the form $\hat \alpha = (u_{n_0}, \cdots, u_{n_{k - 1}})$ for some integer $k \ge 1$, and the $n_i$ are in $\{0, 1, \ldots, q - 1\}$. The unique integer $k \ge 1$ for which $\hat \alpha \in R^k$ is called the \textit{degree of $\hat \alpha$}. In notation, we write $\deg(\hat \alpha)$ for the degree of $\hat \alpha$.

We introduce a semigroup structure on $R^{\omega}$ as follows. For each $\hat \alpha = (u_{n_0}, \ldots, u_{n_{k - 1}})$, $\hat \beta = (u_{n_k}, \ldots, u_{n_s})$, define
$$\hat \alpha \bullet \hat \beta = (u_{n_0}, \ldots, u_{n_{k-1}}, u_{n_k}, \ldots, u_{n_s}).$$
The set $R^{\omega}$ equipped with the binary operation ``$\bullet$'' is clearly a semigroup.

\subsection{Two basic mappings} 
\label{subsec-2maps}

In this subsection, we  introduce two basic mappings, one of which maps elements in $R^{\omega}$ to $V$, and the other maps elements in $R^{\omega}$ to $\bZ_{\ge 0}$. These two maps will play an important role in our definition of generating polynomials for generalized binomial coefficients. 

For each $\hat \alpha = (u_{n_0}, \ldots, u_{n_k}) \in R^{\omega}$, define 
\begin{align}
\label{def-upsilon}
\upsilon_{\hat \alpha} = \sum_{i = 0}^k u_{n_i}\pi^i.
\end{align}
Hence $\upsilon : R^{\omega} \to V$ maps elements from $R^{\omega}$ to $V$

For each $\hat \alpha = (u_{n_0}, \ldots, u_{n_k}) \in R^{\omega}$, note that the $n_i$ are in $\{0, 1, \ldots, q - 1\}$. We define
\begin{align}
\label{def-z}
z_{\hat \alpha} = n_0 + n_1q + \cdots n_kq^k \in \bZ_{\ge 0}.
\end{align}
Thus $z$ maps elements in $R^{\omega}$ to $\bZ_{\ge 0}$.
\begin{remark}
\label{Rem1}
For each $\hat \alpha = (u_{n_0}, \ldots, u_{n_k}) \in R^{\omega}$, if we let 
$$n = n_0 + n_1q + \cdots n_kq^k ,$$
it is easy to see that
$$\upsilon_{\hat \alpha} = \sum_{i = 0}^k u_{n_i}\pi^i = u_n.$$
Furthermore since $z_{\hat \alpha} = n$, we deduce that
\begin{align}
\label{Rem1-e}
\upsilon_{\hat \alpha} = u_{z_{\hat \alpha}}.
\end{align}

\end{remark}

\section{Generating polynomials for generalized binomial coefficients}
\label{section-main-result}

In this section, we prove our main theorem (see Theorem \ref{main-thm} below) which signifies the distribution of generalized binomial coefficients modulo the maximal ideal in a a discrete valuation domain. This result can be viewed as an analogue of a theorem of Garfield and Wilf (see \cite{GW}) in the setting of discrete valuation domains. In contrast with the classical case, the main difficulty in the context of discrete valuation domains is how to construct a well-defined generating polynomial to account for the distribution of generalized binomial coefficients. Indeed a generalized binomial coefficient in a discrete valuation domain $V$ with maximal ideal $\fm$, as recalled at the beginning of this paper, is of the form $F_m(u_n)$, where the sequence of polynomials $(F_h(x))_{h \ge 0} \subset V[x]$ is defined by $(\ref{def-F_n})$, and the sequence $(u_h)_{h \ge 0}$ is defined as in (\ref{def-u_n}). In the classical case, a generating polynomial constructed in the work of Garfield and Wilf (see \cite{GW}) for the binomial coefficients $\binom{m}{n}$ can count, for each fixed nonnegative integer $n$, how many binomial coefficients $\binom{m}{n}$ for $0 \le m \le n$ satisfy certain congruence condition modulo a prime in $\bZ$. The counting is well-defined in the classical case because both variables $m$, $n$ in the binomial coefficient $\binom{m}{n}$ belong in $\bZ$ which is a well-ordered set. In the context of discrete valuation domains, since $u_n$ is an element in $V$ which in general is not a well-ordered set, and $m$ is an integer in $\bZ$ which in general has no relation with $V$, we need to make sense what the counting means for generalized binomial coefficients $F_m(u_n)$. For this, we do not use elements $u_n$ directly, but instead use images of $u_n$ in $\bZ$ via the maps $\upsilon$ and $z$ in Subsection \ref{subsec-2maps}. We begin by describing how counting generalized binomial coefficients work in the setting of discrete valuation domains.

For the rest of this paper, we fix a primitive root, say $\wp$ modulo $\fm$, i.e., $\wp \in V$ is a generator of the cyclic group $(V/\fm)^{\times}$. 

For each $\hat \alpha \in R^{\omega}$ and each $j \in \bZ$, define
\begin{align}
\label{def-S-j}
S_j(\hat\alpha) = \{m \in \bZ \; |\; \text{$0 \le m \le z_{\hat \alpha}$ and $F_m(\upsilon_{\hat \alpha}) \equiv \wp^j \pmod{\fm}$}\}.
\end{align}

By Remark \ref{Rem1}, $\upsilon_{\hat \alpha} = u_{z_{\hat \alpha}}$, and thus one can rewrite $S_j(\hat\alpha)$ as follows:
\begin{align*}
S_j(\hat\alpha) = \{m \in \bZ \; |\; \text{$0 \le m \le z_{\hat \alpha}$ and $F_m(u_{z_{\hat \alpha}}) \equiv \wp^j \pmod{\fm}$}\}.
\end{align*}

For each $n \in \bZ_{\ge 0}$, and each integer $j \in \bZ$, define
\begin{align}
\label{def-epsilon}
\epsilon_j(u_n) = \textbf{card}\{m \in \bZ \; | \; \text{$0 \le m \le n$ and $F_m(u_n) \equiv \wp^j \pmod{\fm}$}\}
\end{align}
where $\textbf{card}(\cdot)$ denotes the cardinality of a set.

The following result follows immediately from Remark \ref{Rem1}.
\begin{proposition}
\label{Prop1}

Let $\hat \alpha \in R^{\omega}$. Then
\begin{align*}
\textbf{card}\left(S_j(\hat \alpha)\right) = \epsilon_j(u_{z_{\hat \alpha}})
\end{align*}
for every $j \in \bZ$.

\end{proposition}

For an integer $n \in \bZ_{\ge 0}$, it is sometimes convenient for the rest of this paper to write the $q$-adic expansion of $n$ in the form
$$n = \sum_{i = 0}^{\infty} n_iq^i,$$
where the $n_i$ are in $\{0, \ldots, q - 1\}$ and all but finitely many $n_i$ are zero. One can use $n$ and its $q$-adic expansion to define certain maps that will be useful in this paper.

To each integer $s \in \bZ_{\ge 0}$, we associate an integer $n(s)$ whose $q$-adic expansion is of the form
\begin{align}
\label{def-n(s)}
n(s) = \sum_{i = 0}^{\infty} n_{i + s}q^i.
\end{align}

On the other hand, to each pair $(r, s) \in \bZ_{\ge 0} \times \bZ_{\ge 0}$, we associate an integer $n(r, s)$ whose $q$-adic expansion is given by
\begin{align}
\label{def-n(r, s)}
n(r, s) = \sum_{i = 1}^r n_{i + s}q^i.
\end{align}
\begin{lemma}
\label{Lem1}

Let $\hat \alpha, \hat \beta \in R^{\omega}$, and let $m \in \bZ_{\ge 0}$. Then
\begin{align*}
F_m(\upsilon_{\hat \alpha \bullet \hat \beta}) \equiv F_{m(\deg(\hat \alpha) -1, 0)}(\upsilon_{\hat \alpha})  F_{m(\deg(\hat \alpha))}(\upsilon_{\hat \beta}) \pmod{\fm}.
\end{align*}

\end{lemma}

\begin{proof}

By Remark \ref{Rem1}, it suffices to prove that
\begin{align*}
F_m(u_{z_{\hat \alpha \bullet \hat \beta}}) \equiv F_{m(\deg(\hat \alpha) -1, 0)}(u_{z_{\hat \alpha}})  F_{m(\deg(\hat \alpha))}(u_{z_{\hat \beta}}) \pmod{\fm}.
\end{align*}

Write
\begin{align*}
\hat \alpha &= (u_{\alpha_0}, \ldots, u_{\alpha_r}), \\
\hat \beta &= (u_{\alpha_{r + 1}}, \ldots, u_{\alpha_{r + s}}),
\end{align*}
where the $\alpha_i$ are in $\{0, \ldots, q - 1\}$, and $r = \deg(\hat \alpha) - 1$, and $s = \deg(\hat \beta)$.

By definition of the operation ``$\bullet$'' (see Subsection \ref{subsec-semigroup}), we can write
\begin{align}
\label{e1-Lem1}
z_{\hat \alpha \bullet \hat \beta} = \sum_{i = 0}^{\infty} \alpha_i q^i,
\end{align}
where for each $i > r + s$, we set $\alpha_i = 0$. By the definition of $u_n$ (see Definition \ref{def-u_n}), 
\begin{align}
\label{e2-Lem1}
u_{z_{\hat \alpha \bullet \hat \beta}} = \sum_{i = 0}^{\infty}u_{\alpha_i} \pi^i.
\end{align}

Let
\begin{align*}
m = \sum_{i = 0}^{\infty} m_iq^i
\end{align*}
be the $q$-adic expansion of $m$, where the $m_i$ are integers in $\{0, \ldots, q - 1\}$, and all but finitely many $m_i$ are zero. By (\ref{e1-Lem1}) and Corollary \ref{Lucas-cor},
\begin{align}
\label{e3-Lem1}
F_m(u_{z_{\hat \alpha \bullet \hat \beta}}) \equiv \prod_{i = 0}^{\infty} F_{m_i}(u_{\alpha_i}) = \left(\prod_{i = 0}^{r} F_{m_i}(u_{\alpha_i})\right)\left(\prod_{j = r + 1}^{\infty} F_{m_j}(u_{\alpha_j})\right) \pmod{\fm}.
\end{align}

By (\ref{def-z}), we know that
\begin{align*}
z_{\hat \alpha} = \sum_{i = 0}^r \alpha_i q^i,
\end{align*}
and it thus follows from Definition \ref{def-u_n} that
\begin{align*}
u_{z_{\hat \alpha}} = \sum_{i = 0}^r u_{\alpha_i} \pi^i.
\end{align*}
On the other hand, by (\ref{def-n(r, s)}), we know that
\begin{align*}
m(\deg(\hat \alpha) -1, 0) = m(r, 0) = \sum_{i = 0}^r m_iq^i.
\end{align*}

By Corollary \ref{Lucas-cor}, we deduce that
\begin{align}
\label{e4-Lem1}
F_{m(\deg(\hat \alpha) -1, 0)}(u_{z_{\hat \alpha}}) \equiv \prod_{i = 0}^r F_{m_i}(u_{\alpha_i}) \pmod{\fm}.
\end{align}

By (\ref{def-n(s)}), and since $\deg(\hat \alpha) = r + 1$, we know that
\begin{align*}
m(\deg(\hat \alpha)) = \sum_{j = 0}^{\infty} m_{r + 1 + j}q^j.
\end{align*}
Since $\alpha_j = 0$ for all $j > r + s$ and $u_0 = 0$, one can write
\begin{align*}
z_{\hat \beta} = \sum_{j = 0}^{\infty} \alpha_{r + 1 + j} q^j, \\
u_{z_{\hat \beta}} = \sum_{j = 0}^{\infty} u_{\alpha_{r + 1 + j}} \pi^j.
\end{align*}
Thus it follows from Corollary \ref{Lucas-cor} and $F_0(x) = 1$ for all $x \in V$ that
\begin{align}
\label{e5-Lem1}
F_{m(\deg(\hat \alpha))}(u_{z_{\hat \beta}}) \equiv \prod_{k = 0}^{\infty} F_{m_{r + 1 + k}}(u_{\alpha_{r +1 + k}}) = \prod_{j = r + 1}^{\infty} F_{m_j}(u_{\alpha_j}) \pmod{\fm}.
\end{align}

We deduce from (\ref{e3-Lem1}), (\ref{e4-Lem1}), and (\ref{e5-Lem1}) that
\begin{align*}
F_m(u_{z_{\hat \alpha \bullet \hat \beta}}) \equiv F_{m(\deg(\hat \alpha) -1, 0)}(u_{z_{\hat \alpha}})  F_{m(\deg(\hat \alpha))}(u_{z_{\hat \beta}}) \pmod{\fm},
\end{align*}
which proves the lemma.

\end{proof}

\begin{corollary}
\label{Cor1}

Let $\hat \alpha, \hat \beta \in R^{\omega}$, and let $\ell, m \in \bZ_{\ge 0}$ such that $0 \le \ell \le z_{\hat \alpha}$ and $0 \le m \le z_{\hat \beta}$. Then
\begin{align*}
F_t(\upsilon_{\hat \alpha \bullet \hat \beta}) \equiv F_{\ell}(\upsilon_{\hat \alpha})F_m(\upsilon_{\hat \beta}) \pmod{\fm},
\end{align*}
where $t = \ell + q^{\deg(\hat \alpha)}m.$

\end{corollary}

\begin{proof}

Write 
\begin{align*}
\hat \alpha = (u_{n_0}, \ldots, u_{n_{r - 1}}),
\end{align*}
where $r = \deg(\hat \alpha)$.

Suppose that $\ell > 0$. Let $\ell = \ell_0 + \ell_1q + \cdots + \ell_sq^s$ be the $q$-adic expansion of $\ell$, where $\ell_s > 0$, and the $\ell_i$ are in $\{0,,1 \ldots, q - 1\}$. By assumption, $\ell \le z_{\hat \alpha} = n_0 + \cdots + n_{r - 1}q^{r - 1}$, and thus there exists the largest integer $0 \le e \le r - 1$ such that $n_e > 0$ and for every $i > e$, $n_i = 0$. Then
$$z_{\hat \alpha} = n_0 + n_1q + \cdots n_eq^e.$$

Since $z_{\hat \alpha} \ge \ell$, we deduce that $e \ge s$. Thus
\begin{align*}
s \le e \le r - 1 < r = \deg(\hat \alpha), 
\end{align*}
and therefore 
$$\ell = \ell_0 + \ell_1 q + \cdots + \ell_sq^s \le q^{s + 1} - 1 < q^{s + 1} \le q^{\deg(\hat \alpha)}.$$
Since $\ell > 0$, $\ell \not\equiv 0 \pmod{q^{\deg(\hat \alpha)}}$. Since 
\begin{align*}
q^{\deg(\hat \alpha)}m \equiv 0 \pmod{q^{\deg(\hat \alpha)}},
\end{align*} 
we deduce that
\begin{align*}
t(\deg(\hat \alpha) - 1, 0) &= \ell.
\end{align*}

If $\ell = 0$, it is trivial that 
\begin{align*}
t(\deg(\hat \alpha) - 1, 0) &= 0 = \ell.
\end{align*}

Let 
$$m = \sum_{j = 0}^{\infty} m_jq^j$$
be the $q$-adic expansion of $m$, and let
$$t = \sum_{k= 0}^{\infty} t_kq^k$$
be the $q$-adic expansion of $t$. Since 
$$t = \ell + q^{\deg(\hat \alpha)}m = \sum_{j = 0}^s \ell_jq^j + \sum_{j = 0}^{\infty}m_j q^{j + \deg(\hat \alpha)},$$
and either $\ell = 0$ or $0 < \ell < q^{\deg(\hat \alpha)}$, the uniqueness of the $q$-adic expansion of $t$ implies that
\begin{align*}
t_k =
\begin{cases}
\ell_k \; \; &\text{if $0 \le k \le s$,} \\
0 \; \; &\text{if $s < k < \deg(\hat \alpha)$,}\\
m_{k - \deg(\hat \alpha)}\; \; &\text{if $k \ge \deg(\hat \alpha)$}.
\end{cases}
 \end{align*}
Thus
\begin{align*}
t(\deg(\hat \alpha)) = \sum_{k = 0}^{\infty}t_{k + \deg(\hat \alpha)}q^k = \sum_{k = 0}^{\infty}m_kq^k = m.
\end{align*}
Using Lemma \ref{Lem1}, we deduce that
\begin{align*}
F_t(\upsilon_{\hat \alpha \bullet \hat \beta})&\equiv F_{t(\deg(\hat \alpha) - 1, 0)}(\upsilon_{\hat \alpha})F_{t(\deg(\hat \alpha))}(\upsilon_{\hat \beta}) \\
&\equiv F_{\ell}(\upsilon_{\hat \alpha})F_m(\upsilon_{\hat \beta})    \pmod{\fm}.
\end{align*}

\end{proof}

We now prove the main lemma in this paper.

\begin{lemma}
\label{main-lem}

Let $\hat{\alpha}, \hat{\beta} \in R^{\omega}$, and $n \in \bZ $, then the set 
$$\bigcup_{j=0}^{q-2} S_j(\hat{\alpha}) \times S_{n-j}(\hat{\beta})$$
is in bijection with $S_n(\hat{\alpha} \bullet\hat{\beta})$.
\end{lemma}

\begin{proof}
We define a mapping $\psi: \bigcup_{j=0}^{q-2} S_j(\hat{\alpha}) \times S_{n-j}(\hat{\beta}) \rightarrow S_n(\hat{\alpha}\bullet \hat{\beta})$ as follows. For a pair $(\ell,m)\in S_j(\hat{\alpha}) \times S_{n-j}(\hat{\beta})$, for some $0 \leq j \leq q-2$, set $t = \ell+q^{\deg{(\hat{\alpha})}}m \in \bZ_{\geq 0}$. Since $0 \leq \ell \leq z_{\hat{\alpha}}$, and $0 \leq m \leq z_{\hat{\beta}}$, $$0 \leq t \leq z_{\hat{\alpha}}+ q^{\deg{(\hat{\alpha})}}z_{\hat{\beta}}.$$

Write $\hat{\alpha}=(u_{n_0}, \ldots , u_{n_{r-1}})$, and $\hat{\beta}=(u_{n_r}, \ldots, u_{n_{r+s}})$. Then $\hat{\alpha} \bullet \hat{\beta} = (u_{n_0}, \ldots, u_{n_{r + s}})$, $z_{\hat{\alpha}}=\sum_{i=1}^{r-1} n_i q^i$, and  $z_{\hat{\beta}}=\sum_{j=0}^{s} n_{r+j}q^j$, and thus 

\begin{align*}
z_{\hat{\alpha}}+q^{\deg{(\hat{\alpha})}}z_{\hat{\beta}} &=z_{\hat{\alpha}}+q^r z_{\hat{\beta}}\\
&= \sum_{i=1}^{r-1} n_i q^i + \sum_{j=0}^{s} n_{r+j}q^{r+j}\\
&= \sum_{j=0}^{r+s} n_j q^j\\
&= z_{\hat{\alpha} \bullet \hat{\beta}}
\end{align*}

Therefore 
\begin{align}
\label{e1-main-lem}
0 \leq t \leq z_{\hat{\alpha} \bullet \hat{\beta}}
\end{align}

By Corollary \ref{Cor1}, and since $(\ell,m)\in S_j(\hat{\alpha}) \times S_{n-j}(\hat{\beta})$, we deduce that 
\begin{align}
\label{e2-main-lem}
F_t(\upsilon_{\hat \alpha \bullet \hat \beta}) &\equiv F_{\ell}(\upsilon_{\hat \alpha})F_m(\upsilon_{\hat \beta}) \pmod{\fm} \nonumber \\
&\equiv \wp^j  \wp^{n-j}=\wp^n \pmod{\fm}
\end{align}

Thus (\ref{e1-main-lem}) and (\ref{e2-main-lem}) imply that $t \in S_n(\hat{\alpha}\bullet \hat{\beta})$. Set 
\begin{align}
\label{e3-main-lem}
\psi (\ell,m)=t=\ell+q^{\deg{(\hat{\alpha})}}m \in S_n(\hat{\alpha}\bullet \hat{\beta}). 
\end{align}

It is obvious that $\psi$ is a well-defined map from $\bigcup_{j=0}^{q-2} S_j(\hat{\alpha}) \times S_{n-j} (\hat{\beta})$ to $S_n(\hat{\alpha}\bullet\hat{\beta})$. We now prove that $\psi$ is a bijection.

We first verify that $\psi$ is surjective. Indeed take an arbitrary element $t \in S_n(\hat{\alpha} \bullet \hat{\beta})$. Set
$$\ell=t(\deg{(\hat{\alpha})}-1,0)\geq 0,$$
and
$$m=t(\deg{(\hat{\alpha})}) \geq 0.$$
By Lemma \ref{Lem1}, we see that 
$$F_{\ell}(\upsilon_{\hat{\alpha}}) F_m(\upsilon_{\hat{\beta}}) \equiv F_t(\upsilon_{\hat{\alpha} \bullet \hat{\beta}}) = \wp^n \pmod{\fm},$$
which implies that both $F_{\ell}(\upsilon_{\hat{\alpha}})$ and $F_m(\upsilon_{\hat{\beta}}$) belong to the multiplicative group $(V / \fm)^{\times}$. Thus $F_{\ell}(\upsilon_{\hat{\alpha}}) \equiv \wp^j \pmod{\fm}$ for some $0 \leq j \leq q-2$, and therefore $F_{m}(\upsilon_{\hat{\alpha}}) \equiv \wp^{n - j} \pmod{\fm}$.

Note that $F_{\ell}(\upsilon_{\hat{\alpha}}) \neq 0$, $F_m(\upsilon_{\hat{\beta}}) \neq 0$, and $\upsilon_{\hat{\alpha}}= u_{z_{\hat{\alpha}}}$, $\upsilon_{\hat{\beta}}=u_{z_{\hat{\beta}}}$ (as in Remark \ref{Rem1}). The definition of $(F_i(x))_{i\geq 0}$ (see (\ref{def-F_n})) implies that $F_i(u_h) = 0$ for any $0 \le h < i$. Thus it follows that 
$$0 \leq \ell \leq z_{\hat{\alpha}} \; \; \text{and} \; \; 0\leq m \leq z_{\hat{\beta}}.$$

Since  $F_{\ell}(\upsilon_{\hat{\alpha}}) \equiv \wp^j \pmod{\fm}$ and  $F_{m(\upsilon_{\hat{\beta}})} \equiv \wp^{n-j} \pmod{\fm}$, we deduce that $\ell \in S_j(\hat{\alpha})$, and $m \in S_{n-j}(\hat{\beta})$. Hence 
$$\psi(\ell, m)=\ell+ q^{\deg{(\hat{\alpha}})}m =t(\deg{(\hat{\alpha})}-1,0)+ q^{\deg{(\hat{\alpha})}} t(\deg{(\hat{\alpha})})=t,$$
and thus $\psi$ is surjective.

Now we show that $\psi$ is injective. Assume $(\ell_1, m_1), (\ell_2, m_2) \in \bigcup_{j=0}^{q-2} S_j(\hat{\alpha}) \times S_{n-j}(\hat{\beta})$ such that 
\begin{align*}
\psi(\ell_1, m_1) &=\ell_1+q^{\deg{(\hat{\alpha})}}m_1=\ell_2 +q^{\deg{(\hat{\alpha})}}m_2= \psi (\ell_2,  m_2).
\end{align*}
Without loss of generality, we can assume that $\ell_1 \ge \ell_2$. Thus,
$$\ell_1-\ell_2 = q^{\deg{(\hat{\alpha})}}(m_2-m_1) \ge 0.$$
If $\ell_1 > \ell_2$, then $m_2 > m_1$, and thus 
$$q^{\deg{(\hat{\alpha})}} \le \ell_1-\ell_2 \leq z_{\hat{\alpha}} < q^{\deg{\hat{\alpha}}},$$ 
which is a contradiction. Thus $\ell_1 = \ell_2$, it follows from the above equation that $m_1 = m_2$, which implies that $\psi$ is injective. Therefore $\psi$ is bijective, and the lemma follows immediately.

\end{proof}

The following result follows immediately from Lemma \ref{main-lem}.

\begin{corollary}
\label{M-cor}

Let $\hat{\alpha}, \hat{\beta} \in R^{\omega}$, and $n \in \bZ$. Then
$$\textbf{card}\left(S_n(\hat{\alpha} \bullet \hat{\beta})\right) = \sum_{j=0}^{q-2} \textbf{card}\left(S_j(\hat{\alpha})\right) \textbf{card}\left(S_{n-j}(\hat{\beta})\right).$$

\end{corollary}

Let $(\epsilon_j (u_n))_{j \in \bZ, n \in \bZ_{\geq 0}}$ be the double sequence, where the $\epsilon_j(u_n)$ is defined by (\ref{def-epsilon}). We are interested in computing the generating function for $(\epsilon_j (u_n))_{j \in \bZ, n \in \bZ_{\geq 0}}$, which signifies the distribution of $F_m(u_n)$ modulo $\fm$ for $0 \le m \le n$. Since the constant $\epsilon_j(u_n)$ counts how many integers between $0$ and $n$ for which the value $F_m(u_n)$ is congruent to $\wp^j$ modulo $\fm$, we only need to compute the values $\epsilon_j(u_n)$ for $0\leq j \leq q-2$. 

For each integer $n\geq 0$, let $\cG_n(x) \in \bZ[x]$ be the polynomial defined by 
$$\cG_n(x) = \sum_{j=0}^{q-2} \epsilon_j (u_n)x^j .$$
It follows from the definition of $\epsilon_j(n)$ that $\cG_0(x)=1$. For each integer $n \ge 0$, the polynomial $\cG_n(x)$ is called the \textit{generating polynomial} for the sequence $(\epsilon_j (u_n))_{j \in \bZ}$.

In order to state the main theorem in this paper, we introduce a finite set of non-negative integers $\{e_j(n)\}_{0 \leq j \leq q-1}$ as follows.\\

For $n=0$, define $e_0(0)=1$ and $e_j(0)=0$ for $1 \leq j \leq q-1$. If $n>0$, write $n$ in the $q$-adic expansion of the form $n=\sum_{i=1}^r n_iq^i$ for some $r \geq 0$, where the $n_i$ are in $\{0,\ldots,q-1\}$, and $n_r >0$. For each $0 \leq j \leq q-1$, we define $e_j(n)$ to be the number of times the integer $j$ occurs in the set $\{n_0, n_1,\ldots, n_r\}$.\\

We now prove our main theorem which can be viewed as an analogue of the Garfield--Wilf theorem for discrete valuation domains.\\

\begin{theorem}
\label{main-thm}

Let $n$ be a nonnegative integer. Then 
\begin{align*}
\cG_n(x) \equiv \prod _{j=0}^{q-1} \cG_j(x)^{e_j(n)} \pmod{(x^{q-1}-1)}.
\end{align*}

\end{theorem}

\begin{proof}
Let $\mathcal{W}$ be the semigroup $\bZ[x]/ (x^{q-1}-1)\bZ[x]$ equipped with the multiplication of polynomials modulo $(x^{q-1}-1)$. Every element in $\mathcal{W}$ can be represented by a polynomial in $\bZ[x]$ of degree $q-2$.\\

We define $\Gamma: R^{\omega} \rightarrow \mathcal{W}$ to be the mapping defined by 
\begin{align*}
\Gamma(\hat{\alpha})=\sum_{m=0}^{q-2} \epsilon_m (u_{z_{\hat{\alpha}}})x^m \pmod{x^{q-1}-1}
\end{align*}
for each $\hat{\alpha} \in R^{\omega}$.

By Remark \ref{Rem1}, one can write
$$
\Gamma(\hat{\alpha})= \sum_{m=0}^{q-2} \epsilon_m (\upsilon_{\hat{\alpha}}) x^m \; \; \text{for each} \; \hat{\alpha} \in R^{\omega}.
$$

We claim that $\Gamma$ is a semigroup homomorphism , i.e., 
$$
\Gamma(\hat{\alpha} \bullet \hat{\beta})=\Gamma(\hat{\alpha})\Gamma(\hat{\beta}) \; \; \text{for any} \; \hat{\alpha}, \hat{\beta} \in R^{\omega}.
$$

Indeed, by Proposition \ref{Prop1} and Corollary \ref{M-cor},
\begin{align*}
\Gamma(\hat{\alpha} \bullet \hat{\beta}) &= \sum_{m=0}^{q-2} \epsilon_m (u_{z_{\hat{\alpha}\bullet\hat{\beta}}})x^m\\
&= \sum_{m=0}^{q-2} \textbf{card}{(S_m (\hat{\alpha}\bullet\hat{\beta}))} x^m\\
&= \sum_{m=0}^{q-2} \sum_{j=0}^{q-2} \textbf{card}{(S_j(\hat{\alpha}))} \textbf{card}{(S_{m-j}(\hat{\beta}))x^m}\\
&= \sum_{m=0}^{q-2} \sum_{j=0}^{q-2} \epsilon_j(u_{z_{\hat{\alpha}}}) \epsilon_{m-j}(u_{z_{\hat{\beta}}})x^m.
\end{align*}

By the definition of $(\epsilon_j (u_n))_{j \in \bZ, n \in \bZ_{\geq 0}}$, it is easy to see that for a fixed integer $n \geq 0$, $\epsilon_j(u_n)$ is periodic of period $q-1$ with respect to the index $j$ (since the group $(\mathcal{V}/ \fm)^{\times}$ is of order $q-1$). Hence,
$$
\epsilon_{m-j} (u_{z_{\hat{\beta}}})= \epsilon_{q-1+(m-j)}(u_{z_{\hat{\beta}}}).
$$
\\

Since $x^{q-1+m} \equiv x^m \pmod{x^{q-1}-1}$,
\begin{align}
\label{e1-main-thm}
\Gamma(\hat{\alpha}\bullet\hat{\beta}) &\equiv \sum_{m=0}^{q-2} \sum_{j=0}^{q-2} \epsilon_j (u_{z_{\hat{\alpha}}}) \epsilon_{q-1+(m-j)} (u_{z_{\hat{\beta}}}) x^{q-1+m} \nonumber\\
&\equiv \sum_{j=0}^{q-2} \epsilon_j (u_{z_{\hat{\alpha}}})x^j \left(\sum_{m=0}^{q-2} \epsilon_{q-1+m-j}(u_{z_{\hat{\beta}}})x^{q-1+m-j}\right) \nonumber\\
& \equiv \sum_{j=0}^{q-2} \epsilon_j (u_{z_{\hat{\alpha}}})x^j \mathcal{A}(j) \pmod{x^{q-1}-1},
\end{align}
where for each $0 \le j \le q - 2$,
\begin{align}
\label{e2-main-thm}
\mathcal{A}(j)= \sum_{m=0}^{q-2} \epsilon_{q-1+m-j} (u_{z_{\hat{\beta}}}) x^{q-1+m-j}.
\end{align}

We claim that $\mathcal{A}(j) \equiv \mathcal{A}(0) \pmod{x^{q-1}-1}$ for every $0 \leq j \leq q-2$. In order to prove this claim, it suffices to prove that 
$$
\mathcal{A}(j) \equiv \mathcal{A}(j+1) \pmod{x^{q-1}-1}, \; \; \text{for each}\; 0\leq j \leq q-3.
$$
\\

Indeed, one sees that
\begin{align*}
\mathcal{A}(j+1) &= \sum_{m=0}^{q-2} \epsilon_{q-1+(m-(j+1))} (u_{z_{\hat{\beta}}})x^{q-1+m-(j+1)}\\
&= \epsilon_{q-2-j}(u_{z_{\hat{\beta}}}) x^{q-2-j} +\sum_{m=1}^{q-2}\epsilon_{q-1+(m-1)-j}(u_{z_{\hat{\beta}}})x^{q-1+(m-1)-j}\\
&= \epsilon_{q-2-j}(u_{z_{\hat{\beta}}}) x^{q-2-j}+ \sum_{m=0}^{q-3} \epsilon_{q-1+(m-j)}(u_{z_{\hat{\beta}}})x^{q-1+m-j}\\
&= \left(\mathcal{A}(j)-\epsilon_{q-1+(q-2)-j}(u_{z_{\hat{\beta}}})x^{q-1+(q-2)-j}\right)+ \epsilon_{q-2-j} (u_{z_{\hat{\beta}}})x^{q-2-j}\\
&= \mathcal{A}(j)+\left(\epsilon_{q-2-j}(u_{z_{\hat{\beta}}})x^{q-2-j}- \epsilon_{2q-3-j}(u_{z_{\hat{\beta}}})x^{2q-3-j}\right)
\end{align*}

Since $(2q-3-j)-(q-2-j)=q-1$, and $(\epsilon_j (u_n))_{j \in \bZ}$ is periodic of period $q-1$ for each fixed integer $n \in \bZ_{\geq 0}$,
$$
\epsilon_{q-2-j}(u_{z_{\hat{\beta}}}) = \epsilon_{2q-3-j} (u_{z_{\hat{\beta}}}).
$$

On the other hand, it is clear that 
$$
x^{2q-3-j} \equiv x^{q-2-j} \pmod{x^{q-1}-1}
$$
since $(2q-3-j)-(q-2-j)=q-1$.\\

Thus $\mathcal{A}(j+1) \equiv \mathcal{A}(j) \pmod{x^{q-1}-1}$ for every $0 \leq j \leq q-3$. \\

Thus $\mathcal{A}(j) \equiv \mathcal{A}(0)$ for all $0 \leq j \leq q-2$. It follows from (\ref{e1-main-thm}) that 
\begin{align*}
\Gamma(\hat{\alpha}\bullet\hat{\beta}) &\equiv \sum_{j=0}^{q-2} \epsilon_j(u_{z_{\hat{\alpha}}}) x^j \mathcal{A}(0)\\
&\equiv \mathcal{A}(0) \sum_{j=0}^{q-2} \epsilon_j (u_{z_{\hat{\alpha}}})x^j \pmod{x^{q-1}-1}
\end{align*}

Note that 
\begin{align*}
\mathcal{A}(0) &= \sum_{m=0}^{q-2} \epsilon_{q-1+m} (u_{z_{\hat{\beta}}}) x^{q-1+m}\\
&\equiv \sum_{m=0}^{q-2} \epsilon_m(u_{z_{\hat{\beta}}})x^m\pmod{x^{q-1}-1}\\
&\equiv \Gamma(\hat{\beta}) \pmod{x^{q-1}-1}.
\end{align*}

Hence
\begin{align}
\label{e3-main-thm}
\Gamma(\hat{\alpha}\bullet\hat{\beta})=\Gamma(\hat{\alpha})\Gamma({\hat{\beta}}),
\end{align}
which proves our claim.

Now we prove that (\ref{e3-main-thm}) implies the theorem. Indeed, the theorem is clearly true for $n = 0$. If $n > 0$, write $n$ in the $q$-adic expansion of the form $n=\sum_{i= 0}^{r} n_i q^i$ for some $r \in \bZ_{\geq 0}$, where $n_r > 0$, and the $n_i$ are in the finite set $\{0, ..., q-1\}$
Set 
\begin{align*}
\hat{\alpha}=(u_{n_0},\ldots, u_{n_r}) \in R^{\omega}. 
\end{align*}

For each $0 \leq i \leq r$, set $\hat{\beta}_i = u_{n_i} \in R^{\omega}$. Then it follows from (\ref{e3-main-thm}) that 
\begin{align*}
\Gamma(\hat{\alpha})=\Gamma(\hat{\beta_0} \bullet \hat{\beta_1} \bullet \cdots\bullet\hat{\beta_r})=\Gamma(\hat{\beta_0}) \cdots \Gamma(\hat{\beta_r}). 
\end{align*}

On the other hand, since $n=z_{\hat{\alpha}}$,
\begin{align*}
\cG_n(x)= \sum_{j=0}^{q-2} \epsilon_j (u_n)x^j = \sum_{j=0}^{q-2} \epsilon_j (u_{z_{\hat{\alpha}}})x^j \equiv \Gamma(\hat{\alpha}) \pmod{x^{q-1}-1}.
\end{align*}

Similarly $\Gamma(\hat{\beta_i}) \equiv \cG_{n_i}(x) \pmod{x^{q-1}-1}$ for each $0 \leq i \leq r$. Hence 
\begin{align*}
\cG_n(x) \equiv \cG_{n_0}(x)...\cG_{n_r}(x) \pmod{x^{q-1}-1}.
\end{align*}

By the definition of $e_j(n)$, it is clear that $\prod_{j=0}^r \cG_{n_j}(x)= \prod_{i=0}^{q-1}\cG_i(x)^{e_i(n)}$, and thus
$$
\cG_n(x)\equiv \prod_{i=0}^{q-1}\cG_i(x)^{e_i(n)} \pmod{x^{q-1}-1},
$$
as required.

\end{proof}

\end{document}